\documentclass[aos,MSNbibl,seceqn,dvips]{arximspdf}

\doi{10.1214/13-AOS1107} 
\volume{41}
\issue{2}
\pubyear{2013}
\firstpage{982}
\lastpage{1004}

\makeatletter

\newtheorem{theorem}{Theorem}[section]
\newtheorem{prop}[theorem]{Proposition}

\newproclaim{ex}[theorem]{Example}
\newproclaim{defn}[theorem]{Definition}
\newproclaim{rem}[theorem]{Remark}

\newcommand{\Real}{\mathbb R}
\newcommand{\A}{\mathcal{A}}
\newcommand{\punt}{\bolds{.}}
\newcommand{\ibs}{\mathbf{i}}
\newcommand{\kappabs}{\bolds{\kappa}}
\newcommand{\nubs}{\bolds{\nu}}
\newcommand{\xibs}{\bolds{\xi}}
\newcommand{\ubs}{\mathbf{u}}
\newcommand{\jbs}{\mathbf{j}}
\newcommand{\ml}{\mathfrak{m}}
\newcommand{\gammabs}{\bolds{\gamma}}
\newcommand{\zbs}{\mathbf{z}}
\newcommand{\ybs}{\mathbf{y}}
\newcommand{\xbs}{\mathbf{x}}
\newcommand{\Tr}{\operatorname{Tr}}
\newcommand{\lambdabs}{\bolds{\lambda}}
\newcommand{\mmodels}{\vdash}

\def\dotplus{\mathrel{\dot{+}}}
\def\Id{I}

\def\given{\mid}
\def\mfk{{\mathfrak K}}
\def\diag{\operatorname{diag}}
\def\var{\operatorname{var}}
\def\cov{\operatorname{cov}}

\makeatother

\begin{document}
\begin{frontmatter}

\title{Natural statistics for spectral samples}
\runtitle{Natural statistics for spectral samples}

\begin{aug}
\author[A]{\fnms{E.} \snm{Di Nardo}\corref{}\ead[label=e1]{elvira.dinardo@unibas.it}},
\author[B]{\fnms{P.} \snm{McCullagh}\ead[label=e2]{pmcc@galton.uchicago.edu}}
\and
\author[A]{\fnms{D.} \snm{Senato}\ead[label=e3]{domenico.senato@unibas.it}}
\runauthor{E. Di Nardo, P. McCullagh and D. Senato}
\affiliation{University of Basilicata, University of Chicago and
University of Basilicata}
\address[A]{E. Di Nardo\\
D. Senato\\
Department of Mathematics\\
\quad Computer Science\\
\quad and Economics\\
University of Basilicata\\
Viale dell'Ateneo Lucano, 10\\
I-85100, Potenza\\
Italy\\
\printead{e1}\\
\hphantom{E-mail: }\printead*{e3}}
\address[B]{P. McCullagh\\
Department of Statistics\\
University of Chicago\\
5734 University Ave \\
Chicago, Illinois 60637\\
USA\\
\printead{e2}} 
\end{aug}

\received{\smonth{6} \syear{2012}}
\revised{\smonth{2} \syear{2013}}

%
\begin{abstract}
Spectral sampling is associated with the group of unitary
transformations acting on matrices in much the same way that simple
random sampling is associated with the symmetric group acting on
vectors. This parallel extends to symmetric functions, $k$-statistics
and polykays. We construct spectral \mbox{$k$-statistics} as unbiased
estimators of cumulants of trace powers of a suitable random matrix.
Moreover we define normalized spectral polykays in such a way that when
the sampling is from an infinite population they return products of
free cumulants.
\end{abstract}

%

\end{frontmatter}

\section{Outline}\label{sec1}
The goals of this paper are threefold.

We first introduce the notion of spectral sampling as an operation on a
finite set of $n$ real numbers $\xbs=(x_1,\ldots, x_n)$ generating a
random set $\ybs=(y_1,\ldots, y_m)$ of $m\le n$ real numbers whose
distribution is determined by $\xbs$. Spectral sampling is not the same
as simple random sampling in the sense that $\ybs$ is not a subset of
$\xbs$, but the parallels are unmistakable and striking. In particular,
there exist symmetric functions $\mfk_\lambda$---analogous to
$k$-statistics and polykays---such that $E(\mfk_\lambda(\ybs)
\given\xbs) = \mfk_\lambda(\xbs)$. In other words, the average value of
$\mfk_\lambda(\cdot)$ for spectral samples $\ybs$ taken from $\xbs$ is
equal to $\mfk_\lambda(\xbs)$. The first goal is to obtain explicit
expressions for these spectral $k$-statistics, which is done in
Sections~\ref{sec3}--\ref{sec5} using symbolic umbral techniques.

The second goal is to elucidate some of the concepts associated with
freeness---free probability and free cumulants---in terms of
spectral sampling and spectral $k$-statistics.
For this purpose, spectral sampling may be viewed as a restriction
operation $X\mapsto Y$
from a freely randomized Hermitian matrix of order $n$
into a freely randomized Hermitian matrix of order $m\le n$,
and each spectral $k$-statistic is class function depending only on the
matrix eigenvalues.
In essence, the spectral $k$-statistics tell us which spectral
properties are preserved on average by freely randomized matrix restriction.
For example,\vadjust{\goodbreak} $\mfk_{(1)}(\xbs) = \bar{\xbs}$ tells us that the
eigenvalue average
is preserved.
Likewise, if $k_2$ denotes the usual sample variance with divisor $n-1$,
the second spectral statistic $\mfk_{(2)}(\xbs) = k_2(\xbs)/(n+1)$
tells us that the eigenvalue sample variance is not
preserved, but is, on average, proportional to the sample size plus one.

Finally, by considering the limit as $n\to\infty$, we show that the
normalized spectral $k$-statistics are related to free cumulants
in much the same way that polykays are related to ordinary cumulants.

\section{Spectral sampling}\label{sec2}
\subsection{Definition}\label{sec2.1}
A random Hermitian matrix $A$ of order $n$ is said to be
\textit{freely randomized} if its distribution is invariant under
unitary conjugation,
that is, $A \sim G A G^\dag$ for each unitary $G$.
In particular, if $H$ is uniformly distributed with
respect to Haar measure on the group of unitary matrices of order $n$,
$HAH^\dag$ is freely randomized.
If $A$ is freely randomized, each leading sub-matrix is also freely randomized.

Let $\xbs=(x_1,\ldots, x_n)$ be given real numbers,
let $X=\diag(\xbs)$ be the associated diagonal matrix
and let $HXH^\dag$ be the freely randomized matrix.
The sample matrix $Y$ is the leading $m\times m$ sub-matrix
in the freely randomized matrix, that is, $Y = (HXH^\dag)_{[m\times m]}$.

%
%
\begin{defn}[(Spectral sample)] \label{Def2}
The set of eigenvalues $\ybs=(y_1,\ldots,\break y_m) \in\Real^m$ of
the $m \times m$
Hermitian random matrix $Y = (H X H^\dag)_{[m \times m]}$ is called
a \textit{spectral sample} of size $m$ from $\xbs$.
\end{defn}

For $m=n$, the distribution is uniform with the same weight $1/n!$ on
all permutations
$\sigma\in{\mathfrak S}_n$; that is, $\ybs$ is a random permutation
of $\xbs$.
For $m < n$, however, the distribution in $\Real^m$ is nonatomic,
so the sample values $\ybs$ do not ordinarily occur among
the components of $\xbs$.

If the group of unitary transformations in the preceding definition
were replaced
by a sub-group, the sampling distribution would be altered accordingly.
The most obvious subgroups are the group of orthogonal transformations
and the group of permutations $[n]\to[n]$;
in each case there is an associated family of spectral functions
such that $E(\mfk_\lambda(\ybs)\given\xbs) = \mfk_\lambda(\xbs)$.
In particular, if $H$ is a uniform random permutation,
$\ybs$ is a simple random sample of size $m$ taken from $\xbs$, and
the associated spectral functions are the classical $k$-statistics
due to Fisher~\cite{Fisher} and the polykays due to Tukey
\cite{Tukey1}.
%
%
\begin{rem}
Within image compression~\cite{Raviraj}, the random Hermitian
matrix $Y$ in Definition~\ref{Def2}
is called a two-dimensional Haar transform. More generally, if $X$ is a
full matrix whose entries are the pixels ranging from $0$ (black) to
$255$ (white), then $Y$ contains reduced information extracted from $X$
via the rectangular Haar matrix $H$. Similar transformations are
employed also within classification, document analysis, hardware
implementation and are known as downsampling of a vector or a matrix
\cite{Strang}.\vadjust{\goodbreak}
\end{rem}
%
\subsection{Natural statistics}\label{sec2.2}
For present purposes, a statistic $T$ is a collection of functions
$T_n\colon\Real^n\to\Real$ such that $T_m(\ybs)$ and $T_n(\xbs)$
are defined for all samples sufficiently large.
For example, the usual sample variance is defined for $n\ge2$,
while the sample skewness is defined for $n\ge3$.

%
%
\begin{defn}[(Natural statistic)] \label{def1}
A statistic $T$ is said to be \textit{natural} if, for each $m\le n$,
the average value of $T_m(\cdot)$ over random samples $\ybs$
drawn from $\xbs$ is equal to $T_n(\xbs)$.
In symbols,
\[
E \bigl(T_m(\ybs)\given\xbs\bigr)=T_n(\xbs)
\]
for each $m \leq n$.
\end{defn}
Obviously, the definition depends on what it means for
$\ybs$ to be a random sample drawn from $\xbs$,
that is, the choice of group in Definition~\ref{Def2}.
Thus, a statistic that is natural with respect to simple random sampling
(a $U$-sta\-tistic)
is not, in general, natural with respect to spectral sampling.

In Tukey~\cite{Tukey1}, such functions were said to be ``inherited on the
average.'' The key point in Definition~\ref{def1}
is that a natural statistic is not a single function in isolation, but a
list of functions $\{T_n\dvtx \Real^n \rightarrow\Real\}$.
It is the property of inheritance that gives these functions a common
interpretation independent of the sample size. One might be inclined to
think that
inheritance is no different from unbiasedness relative to a model with
exchangeably distributed components. However, unbiasedness of $T_n$
does not
imply the inheritance property, nor does inheritance imply that the
statistic has a
limit or that its expectation exists. Unbiasedness in parametric models
is a
property of individual functions $T_n$, whereas inheritance is a
property of the sequence.

For $m=n$, inheritance implies that each $T_n$ is a symmetric function:
$T_n(\xbs)$ is
equal to the average of the values on the permutations of $\xbs$.
Tukey~\cite{Tukey1} proved that the symmetric functions
%
%
\begin{eqnarray}\label{augnorm}
\tilde{{\mathfrak a}}_{r,n}(\xbs) & = & \frac{1}{n} \sum
_i x_i^r,\qquad \tilde{{\mathfrak
a}}_{rs,n}(\xbs) = \frac{1}{(n)_2} \sum_{i\ne j}
x_i^r x_j^s,
\nonumber\\[-8pt]\\[-8pt]
\tilde{{\mathfrak a}}_{rst,n}(\xbs) & = & \frac{1}{(n)_3} \sum
_{i
\ne j \ne k} x_i^r x_j^s
x_k^t,\qquad \ldots\nonumber
\end{eqnarray}
defined, respectively, for $n \geq1, n \geq2$ and so on, are
natural with respect to simple random sampling.
Here and elsewhere $(n)_r = n(n-1)\cdots(n-r+1)$ denotes the
descending factorial function.
Ordinarily, we suppress the index $n$ and write
$\tilde{{\mathfrak a}}_{rs}(\xbs)$ instead of
$\tilde{{\mathfrak a}}_{rs,n}(\xbs)$,
the value of $n$ being inferred from the argument $\xbs\in\Real^n$.
The unnormalized polynomials
%
%
\begin{eqnarray}\label{aug}
{\mathfrak a}_{r}(\xbs)&=& \sum_i
x_i^r,\qquad {\mathfrak a}_{rs}(\xbs) = \sum
_{i\ne j} x_i^r
x_j^s,\nonumber\\[-8pt]\\[-8pt] {\mathfrak a}_{rsk}(\xbs) &=& \sum
_{i \ne j \ne t} x_i^r
x_j^s x_t^k,\qquad \ldots\nonumber
\end{eqnarray}
are the well-known \textit{augmented symmetric functions}
\cite{Kendall}.

Every expression which is a polynomial, symmetric and inherited on the
average can
be written as a linear combination of the statistics in (\ref
{augnorm}) with
coefficients that do not depend on the size of the set~\cite{Tukey1}.
Consequently each linear combination, with scalar coefficients
independent of $n$, also has the inheritance property, as happens, for
example, for
$U$-statistics. The combinations that have proved to be most useful for
statistical
purposes are the $k$-statistics due to Fisher~\cite{Fisher} and the
polykays due to Tukey
\cite{Tukey1,Tukey2}, defined as follows:
\begin{eqnarray*}
k_{(1)} &=& \tilde{{\mathfrak a}}_{(1)};
\\
k_{(1^2)} &=& \tilde{{\mathfrak a}}_{(1^2)},\qquad k_{(2)} =
\tilde{{\mathfrak a}}_{(2)} - \tilde{{\mathfrak a}}_{(1^2)};
\\
k_{(1^3)} &=& \tilde{{\mathfrak a}}_{(1^3)},\qquad k_{(12)} =
\tilde{{\mathfrak a}}_{(12)} - \tilde{{\mathfrak a}}_{(1^3)},\qquad
k_{(3)} = \tilde{{\mathfrak a}}_{(3)} - 3 \tilde{{\mathfrak
a}}_{(12)} + 2 \tilde{{\mathfrak a}}_{(1^3)};
\\
k_{(1^4)} &=& \tilde{{\mathfrak a}}_{(1^4)},\qquad k_{(1^2 2)} =
\tilde{{\mathfrak a}}_{(1^2 2)} - \tilde{{\mathfrak a}}_{(1^4)},\qquad
k_{(13)} = \tilde{{\mathfrak a}}_{(13)} - 3 \tilde{{\mathfrak
a}}_{(1^2 2)} + 2 \tilde{{\mathfrak a}}_{(1^4)};
\\
k_{(2^2)} &=& \tilde{{\mathfrak a}}_{(2^2)} - 2 \tilde{{\mathfrak
a}}_{(1^2 2)} + \tilde{{\mathfrak a}}_{(1^4)},\\
k_{(4)} &=&
\tilde{{\mathfrak a}}_{(4)} - 4 \tilde{{\mathfrak a}}_{(13)} - 3
\tilde{{\mathfrak a}}_{(2^2)} + 12 \tilde{{\mathfrak a}}_{(1^2 2)} - 6
\tilde{{\mathfrak a}}_{(1^4)}.
\end{eqnarray*}
The single index $k$'s are the $k$-statistics;
the multi-index $k$'s are the polykays.
For a sample of i.i.d. variables,
each $k$-statistic is an unbiased estimator of the population cumulant,
and each polykay is an unbiased estimator of cumulant products.
The \textit{degree} of each $k$ is the sum of the subscripts.
The set of natural polynomial statistics of degree $i$ is a vector space,
of dimension equal to the number of partitions of the integer $i$,
spanned by the $k$'s of degree $i$.
%

\section{Moment symbolic method}\label{sec3}
\textit{Univariate case}. The moment symbolic method relies on the classical
umbral calculus introduced by Rota and Taylor in 1994~\cite{SIAM},
which has been developed and refined in a series of papers starting from
\cite{Dinsen,Dinardoeurop}.
The result is a calculus in which certain symbols represent scalar
or polynomial sequences, thereby reducing the overall computational apparatus.
We now review the key components.

Let $R$ be the real or complex field whose elements are called scalars.
An umbral calculus consists of a generating set $\A=\{\alpha,\beta,
\ldots\}$, called the \textit{alphabet}, whose elements are named
\textit{umbrae}, a polynomial ring $R[\A]$ and a linear functional
$E\colon R[\A]\to R$ called evaluation. The linear functional is such
that $E[1]=1$ and
%
%
\begin{equation}\label{iii}\quad
E\bigl[\alpha^{i} \beta^{j} \cdots\gamma^{k}
\bigr] = E\bigl[\alpha^{i}\bigr] E\bigl[\beta^{j}\bigr]
\cdots E\bigl[\gamma^{k}\bigr]\qquad {\mbox{(uncorrelation property)}}
\end{equation}
for any set of distinct umbrae in $\A$ and for $i,j,k$
nonnegative integers.
To each umbra $\alpha\in\A$ there corresponds a
sequence of scalars\vadjust{\goodbreak} $a_i = E[\alpha^i]$ for $i=0,1,\ldots$ such that $a_0=1$.
The scalar $a_i$ is called the $i$th moment of $\alpha$.
Indeed any scalar random variable possessing finite moments can be
represented by an umbra.
A~scalar sequence $\{a_i\}$ with $a_0=1$ is said to be represented by an
umbra $\alpha$ if $E[\alpha^i] = a_i$ for $i=0,1,\ldots\,$.
%
%
\begin{ex}
The sequence $1,0,0,\ldots$ is umbrally represented by
the \textit{augmentation umbra} $\varepsilon$, and
$1,1,1,\ldots$ is umbrally represented by the \textit{unity umbra}~$u$.
These are the umbral versions of two degenerate random variables
such that $P(X=0)=1$ and $P(Y=1)=1$.
The sequence of moments of a unit Poisson random variable is
umbrally represented by the \textit{Bell umbra} $\beta$.
This umbra plays a fundamental role in the symbolic method, as we will
see later.
Its $i$th moment is the Bell number, which is
the coefficient of $z^i/i!$ in the Taylor expansion of $\exp(e^z-1)$.
\end{ex}
Since an umbra is a formal object, questions involving
the moment problem are not taken into account.
Indeed, not every umbra corresponds to a real-valued random variable.
%
%
\begin{ex}
The sequence $1,1,0,0,\ldots$ is represented by the \textit{singleton}
umbra $\chi$. Its variance $E[\chi^2]-E[\chi]^2=-1$ is negative, so
there is no real-valued random variable corresponding to $\chi$.
Nevertheless this umbra plays a fundamental role in dealing with
cumulant sequences, as we will see later.
\end{ex}

It is always possible to make the alphabet large enough so that,
to each scalar sequence $\{a_i\}$, there corresponds an umbra $\alpha
\in\A$,
which is not necessarily unique.
The same applies to identically distributed random variables.
Two umbrae $\alpha$ and $\gamma$ having the same moment sequence are
called similar,
in symbols $\alpha\equiv\gamma$, and $\A$ contains an unlimited
supply of distinct umbrae
similar to $\alpha$, usually denoted by $\alpha', \alpha'',\ldots\,$.
If the sequence $\{a_i\}$ is umbrally represented by $\alpha$, then
\begin{eqnarray*}
&\mbox{the sequence } \bigl\{2^i a_i \bigr\}
\mbox{ is represented by } \alpha+ \alpha= 2 \alpha,&
\\
&\mbox{the sequence } \displaystyle \Biggl\{ \sum_{k=0}^i
\pmatrix{i
\cr
k} a_k a_{i-k} \Biggr\} \mbox{ is
represented by } \alpha+ \alpha^{\prime}.&
\end{eqnarray*}
An expression such as $2 \alpha$ or $\alpha+ \alpha^{\prime}$ is an example
of an umbral polynomial, that is, a polynomial $p \in R[\A]$ in the
umbrae of $\A$.
The support of an umbral polynomial is the set of all umbrae that occur
in it.
So the support of $\alpha+ \alpha^{\prime}$ is $\{\alpha,\alpha
^{\prime}\}$,
and the support of $2 \alpha$ is $\{\alpha\}$. The formal power series
%
%
\begin{equation}\label{gf1}
e^{\alpha z} = u + \sum_{i \geq1}
\alpha^i \frac{z^i}{i!} \in R[\A] [[z]]
\end{equation}
is the \textit{generating function} of the umbra $\alpha$.
Moreover, each exponential formal power series
%
%
\begin{equation}\label{formpow}
f(z) = 1 + \sum_{i \geq1} a_i
\frac{z^i}{i!} \in R[[z]]
\end{equation}
can be umbrally represented by a formal power series (\ref{gf1}) in
$R[\A][[z]]$~\cite{Taylor1}.
In fact, if the sequence $1,a_1,a_2,\ldots$ is umbrally represented by
$\alpha$,
the action of evaluation $E$ can be extended coefficient-wise to formal
power series
(\ref{gf1}), so that $E[e^{\alpha z}]=f(z)$. For clarity we denote
the generating function of $\alpha$ by $f(\alpha, z) = E[e^{\alpha
z}]$. Therefore
$\alpha\equiv\alpha^{\prime}$ if and only if $f(\alpha,z)=f(\alpha
^{\prime},z)$.

The first advantage of umbral notation is the representation
of operations on generating functions with operations on umbrae.
For example, multiplication of exponential generating functions
is umbrally represented by the sum of the corresponding umbrae, that is,
\[
f(\alpha+\gamma,z) = f(\alpha,z)f(\gamma,z).
\]
Therefore $f(\alpha,z)^2$ is the generating function of $\alpha+
\alpha^{\prime}$, which is different from the generating function
$f(\alpha,2 z)$ of $2 \alpha$. The sum of generating functions is
represented by the auxiliary umbra $\alpha\dotplus\gamma$, named the
\textit{disjoint sum} of two umbrae, that is,
\[
f(\alpha\dotplus\gamma,z) = f(\alpha,z) + f(\gamma,z) -1,
\]
so that $E[(\alpha\dotplus\gamma)^i]=E[\alpha^i] + E[\gamma^i]$
for all positive integers $i$.
Then $2 f(\alpha,z) - 1$ is the generating function of
$\alpha\dotplus\alpha$ or $\alpha\dotplus\alpha^{\prime}$,
and $\alpha\dotplus\alpha\equiv\alpha\dotplus\alpha^{\prime}$.

It is also possible to compose generating functions and to
represent the composition as the generating function of an umbra.
First consider $n$ uncorrelated umbrae
$\alpha',\alpha'',\ldots,\alpha'''$ similar to $\alpha$ and take their
sum: the resulting umbra $\alpha' + \alpha'' + \cdots+ \alpha'''$,
denoted by $n \punt\alpha$, is called the \textit{dot product} of the
integer $n$ and the umbra $\alpha$. Its generating function is $f(n
\punt\alpha,z) = [f(\alpha,z)]^n$ and the moments are~\cite{Bernoulli}
%
%
\begin{equation}\label{sum3}
E\bigl[(n \punt\alpha)^i\bigr] = \sum_{\lambda\vdash i}
d_{\lambda} (n)_{l({\lambda})} a_{\lambda} \qquad\mbox{with } d_{\lambda}
= \frac{i!}{(1!)^{r_1} r_1!
(2!)^{r_2} r_2! \cdots},
\end{equation}
where $\lambda$ is a partition of the integer $i$ into $l(\lambda)$ parts,
and
$a_{\lambda}=a_1^{r_1} a_2^{r_2} \cdots$ is the moment product \cite
{Dinardoeurop}.
The right-hand side of (\ref{sum3}) corresponds to
$E[(X_1 + \cdots+ X_n)^i]$ with $X_1,\ldots, X_n$ i.i.d. with
moment sequence represented by the umbra $\alpha$.
In (\ref{sum3}), set $E[(n \punt\alpha)^i]=q_i(n)$, which
is a polynomial of degree $i$ in $n$.
If the integer $n$ is replaced by any umbra $\gamma\in\A$,
and $(\gamma)_j= \gamma(\gamma- 1) \cdots(\gamma-j + 1)$
denotes the descending factorial polynomial, then we have
$q_i(\gamma) = \sum_{\lambda\vdash i} (\gamma)_{l({\lambda})}
d_{\lambda} a_{\lambda}$.
The symbol $\gamma\punt\alpha$ such that $E[(\gamma\punt\alpha
)^i] = E[q_i(\gamma)]$
is called the \textit{dot-product} of the umbrae $\alpha$ and $\gamma$.
This last equality could be rewritten by using the umbral
equivalence $\simeq$
such that $p \simeq q$ iff $E[p] = E[q]$ with $p,q \in R[\A]$.
Then we have $(\gamma\punt\alpha)^i \simeq q_i(\gamma)$.
More generally, the umbral equivalence turns out to be useful in
dealing with
umbral polynomials with nondisjoint supports as we will see later.
The replacement of the integer $n$ with the umbra $\gamma$ is an example
of the main device employed in the symbolic method, allowing us
to represent more structured moment sequences starting from (\ref{sum3}).
Observe that we move from the generating function $[f(\alpha,z)]^n$
to the generating function
$f(\gamma\punt\alpha, z) = f(\gamma, \log[f(\alpha,z)])$, which
is not yet
the composition of $f(\alpha,z)$ and $f(\gamma,z)$.
For this purpose, the umbra $\alpha$ in the dot product $\gamma\punt
\alpha$
has to be replaced by a dot product involving
the Bell umbra, that is, $\beta\punt\alpha$.
The dot product $\beta\punt\alpha$ is called the
$\alpha$-\textit{partition umbra} with generating function
$f(\beta\punt\alpha,z) = \exp(f(\alpha,z) - 1)$.
A special property which we use later is
%
%
\begin{equation}\label{props}
\beta\punt(\alpha\dotplus\gamma) \equiv\beta\punt\alpha+ \beta\punt
\gamma.
\end{equation}
The symbol $\gamma\punt(\beta\punt\alpha)$ has generating function
which is
the composition of $f(\alpha,z)$ and $f(\gamma,z)$
%
%
\begin{equation}\label{composition}
f\bigl(\gamma\punt(\beta\punt\alpha), z\bigr) = f\bigl(\gamma, f(\alpha,z)-1
\bigr).
\end{equation}
Parenthesis can be avoided since $\gamma\punt(\beta\punt\alpha)
\equiv(\gamma\punt\beta) \punt\alpha$.
The moments are
%
%
\begin{equation}\label{comp}
E\bigl[(\gamma\punt\beta\punt\alpha)^i\bigr] = \sum
_{\lambda\vdash i} g_{l({\lambda})}^{} d_{\lambda}
a_{\lambda},
\end{equation}
where $\{g_{i}\}$ are the moments of the umbra $\gamma$~\cite{Dinardoeurop}.
%
%
\begin{ex} \label{comppoiss}
The composition umbra arises naturally in connection with random sums
$X_1+ \cdots+X_N$, where the $X$'s are i.i.d., and $N$ is distributed
independently of $X$. The cumulant generating function of the sum is
the composition $K_N(K_X(t))$ of the two generating functions. In
probability theory, $N$ is necessarily integer-valued, but there is no
such constraint on the umbra $\gamma$.
\end{ex}
Strictly connected to the composition umbra is the \textit{compositional
inverse umbra} $\alpha^{\langle-1\rangle}$ of an umbra $\alpha
$, such that
%
%
\begin{equation}\label{compinv}
\alpha^{\langle-1\rangle} \punt\beta\punt\alpha\equiv\chi\equiv\alpha
\punt\beta
\punt\alpha^{\langle-1\rangle}.
\end{equation}
A special compositional inverse umbra is $u^{\langle-1\rangle}$,
with $u$ the unity umbra,
having generating function
%
%
\begin{equation}\label{gencompinvu}
f\bigl(u^{\langle-1\rangle},z\bigr)=1+\log(1+z)
\end{equation}
so that its $i$th moment is
%
%
\begin{equation}\label{momcompinvu}
E\bigl[\bigl(u^{\langle-1\rangle}\bigr)^i\bigr]=(-1)^{i-1}
(i-1)!.
\end{equation}

\textit{Multivariate case.} Let $\{\nu_1,\ldots, \nu_m\}$ be a set of
umbral monomials
with support not necessarily disjoint. A vector sequence $\{g_{\ibs}\}
_{\ibs\in\mathbb{N}_0^m} \in R$, with $g_{\ibs} = g_{i_1, i_2,\ldots,
i_m}$ and $g_{\mathbf0} = 1$, is
represented by the $m$-tuple $\nubs=(\nu_1,\ldots,\nu_m)$ if
%
%
\begin{equation}\label{multmoments}
E\bigl[\nubs^{\ibs}\bigr] = g_{\ibs}
\end{equation}
for each multi-index $\ibs\in\mathbb{N}_0^m$. The elements $\{
g_{\ibs}\}_{\ibs\in\mathbb{N}_0^m}$ in (\ref{multmoments}) are
called \textit{multivariate moments} of $\nubs$.
%
%
\begin{rem}
Within random variables, the $m$-tuple $\nubs=(\nu_1,\ldots,\nu
_m)$ corresponds to a random vector $(X_1,\ldots, X_m)$. If $\{\nu
_i\}_{i=1}^m$ are uncorrelated
umbrae, then $g_{\ibs}= E[\nu_1^{i_1}] \cdots E[\nu_n^{i_m}]$, and
we recover the univariate symbolic method. The same happens
if $\{\nu_i\}_{i=1}^m$ are umbral monomials with disjoint supports.
\end{rem}
As done in (\ref{gf1}), the generating function of $\nubs$ is the
formal power series
%
%
\begin{equation}\label{gf}
e^{\nu_1 z_1 + \cdots+ \nu_m
z_m} = u + \sum_{k \geq1} \sum
_{{|\ibs|=k}} \nubs^{\ibs}\frac
{\zbs^{\ibs}}{\ibs!} \in R[\A]
[[z_1,\ldots,z_m]]
\end{equation}
with $\zbs= (z_1,\ldots, z_m), |\ibs|=i_1 + \cdots+ i_m$ and $\ibs
!=i_1! \cdots i_m!$.
If the sequence $\{g_{\ibs}\}$ is umbrally represented by $\nubs$ and
has (exponential) generating function
%
%
\begin{equation}\label{genfun1}
f(\zbs) = 1 + \sum_{k \geq1} \sum
_{|\ibs|=k} g_{\ibs} \frac
{\zbs^{\ibs}}{\ibs!},
\end{equation}
then $E[e^{\nu_1 z_1 + \cdots+ \nu_m z_m}] = f(\zbs)$. Taking into
account (\ref{multmoments}), the generating function in (\ref{genfun1})
is denoted by $f(\nubs,\zbs)$.
Two umbral vectors $\nubs_1$ and $\nubs_2$ are said to be
\textit{similar}, in symbols $\nubs_1 \equiv\nubs_2$, if and only if
$f(\nubs_1,\zbs)=f(\nubs_2,\zbs)$, that is, $E[\nubs_1^{\ibs
}]=E[\nubs_2^{\ibs}]$ for all $\ibs\in\mathbb{N}_0^m$. They are said to
be \textit{uncorrelated} if and only if $E[\nubs_1^{\ibs}
\nubs_2^{\jbs}]= E[\nubs_1^{\ibs}]E[\nubs_2^{\jbs}]$ for all $\ibs,
\jbs\in\mathbb{N}_0^m$.

An equation analogous to (\ref{sum3}) could be given for the
multivariate case,
provided that integer partitions are replaced with multi-index
partitions~\cite{Statcomp}.
A partition of a multi-index $\ibs$ is a composition $\lambdabs$,
whose columns
are in lexicographic order, in symbols $\lambdabs\mmodels\ibs$.
A composition $\lambdabs$ of a multi-index $\ibs$ is a matrix
$\lambdabs= (\lambda_{ij})$ of nonnegative integers and
with no zero columns such that $\lambda_{r1}+\lambda_{r2}+\cdots
+\lambda_{rk}=i_r$ for
$r=1,2,\ldots,n$. The number of columns of $\lambdabs$ is the length
of $\lambdabs$
and denoted by~$l(\lambdabs)$. As for integer partitions, the notation
$\lambdabs= (\lambdabs_{1}^{r_1}, \lambdabs_{2}^{r_2}, \ldots)$
means that in the matrix $\lambdabs$ there are $r_1$ columns equal to
$\lambdabs_{1}$,
$r_2$ columns equal to $\lambdabs_{2}$ and so on, with $\lambdabs_{1} <
\lambdabs_{2} < \cdots\,$. We set
$\ml(\lambdabs)=(r_1, r_2,\ldots)$.
The dot-product $n \punt\nubs$ of a nonnegative integer $n$ and a
$m$-tuple $\nubs$ is an auxiliary umbra
denoting the summation $\nubs^{\prime} + \nubs^{\prime\prime} +
\cdots+ \nubs^{\prime\prime\prime}$ with
$\{\nubs^{\prime}, \nubs^{\prime\prime},\ldots, \nubs^{\prime
\prime\prime}\}$ a set of $n$ uncorrelated and similar $m$-tuples.
For $\ibs\in{\mathbb N}^m_0$ and $m$-tuples $\nubs$ of umbral
monomials, we have
%
%
\begin{equation}\label{eq15}
E\bigl[(n \punt\nubs)^{\ibs}\bigr] = \sum_{\lambdabs\mmodels\ibs}
\frac{\ibs!}{\ml(\lambdabs)! \lambdabs!} (n)_{l(\lambdabs)}
g_{\lambdabs},\vadjust{\goodbreak}
\end{equation}
where the sum is over all partitions $\lambdabs= (\lambdabs
_{1}^{r_1}, \lambdabs_{2}^{r_2}, \ldots)$ of the multi-index $\ibs,
g_{\lambdabs} = g_{\lambdabs_1}^{r_1} g_{\lambdabs_2}^{r_2} \cdots$
and $g_{\lambdabs_i} = E[\nubs^{\lambdabs_i}]$. The sequence\vspace*{1pt} in
(\ref{eq15}) represents moments of a sum of i.i.d. random vectors
with sequence of moments $\{g_{\ibs}\}$. If we replace the integer $n$
in (\ref{eq15}) with the dot-product
$\alpha\punt\beta$ we get the auxiliary umbra $\alpha\punt\beta
\punt\nubs$
representing the sequence of moments
%
%
\begin{equation}\label{eq16}
E\bigl[(\alpha\punt\beta\punt\nubs)^{\ibs}\bigr] = \sum
_{\lambdabs
\mmodels\ibs} \frac{\ibs!}{\ml(\lambdabs)! \lambdabs!} a_{l(\lambdabs)}
g_{\lambdabs},
\end{equation}
where the sequence $\{a_i\}$ is umbrally represented by $\alpha$.
In particular the generating function of the auxiliary umbra $\alpha
\punt\beta\punt\nubs$ turns to be the composition
of the univariate generating function $f(\alpha,z)$ and the
multivariate generating function $f(\nubs,\zbs)$
%
%
\begin{equation}\label{multivariate1}
f(\alpha\punt\beta\punt\nubs,\zbs)=f\bigl[\alpha, f(\nubs,\zbs)-1\bigr].
\end{equation}
From Example~\ref{comppoiss}, the umbra $\alpha\punt\beta\punt
\nubs$ is a generalization of a multivariate compound randomized
Poisson random vector.
In the next section, we show how this umbra allows us to write a
formula for multivariate cumulants involving multi-index partitions.
More details on the symbolic composition of multivariate formal power
series can be found in~\cite{Multivariate}.

\section{Formal cumulants}\label{sec4}
\subsection{Definition}\label{sec4.1}
Among the sequences of numbers related to a real-valued random variable,
cumulants play a central role.
Whether or not the sequence $\{a_i\}$ corresponds to the moments of
some distribution,
we define cumulants $\{c_i\}$ by the following equation:
%
%
\begin{equation}\label{formalcumulants}
1 + \sum_{i \geq1} a_i \frac{z^i}{i!} =
\exp\biggl\{ \sum_{i \geq
1} c_i
\frac{z^i}{i!} \biggr\}.
\end{equation}
If $\alpha$ is an umbra representing the sequence $\{a_i\}$, and
$\kappa_{{\alpha}}$
is an umbra representing the sequence $\{c_i\}$, then by comparing
(\ref{formalcumulants}) with (\ref{composition}) we have
%
%
\begin{equation}\label{cumuniv}
\alpha\equiv u \punt\beta\punt\kappa_{{\alpha}},
\end{equation}
since $f(u,z)=\exp(z)$. The umbra $\kappa_{{\alpha
}}$ is called
the $\alpha$-\textit{cumulant umbra}~\cite{Bernoulli} and is such that
%
%
\begin{equation}\label{genfuncum}
f(\kappa_{{\alpha}}, z) = 1 + \log\bigl(f(\alpha,
z)\bigr).
\end{equation}
By comparing (\ref{genfuncum}) with (\ref{composition}) and (\ref
{gencompinvu}), we have
%
%
\begin{equation}\label{fundcum}
\kappa_{{\alpha}} \equiv u^{\langle-1\rangle} \punt
\beta\punt\alpha.
\end{equation}
Since $u^{\langle-1\rangle} \punt\beta\punt u \equiv
u^{\langle-1\rangle} \punt\beta\equiv\chi$, then equivalence
(\ref{fundcum}) reduces to
\[
\kappa_{{\alpha}} \equiv\chi\punt\alpha.
\]

The algebraic properties of cumulants can be formalized as
%
%
\begin{eqnarray}
\label{propcum1}
\mbox{\textit{Homogeneity}}\quad\quad\hspace*{3pt} \chi\punt(a \alpha) &\equiv&
a (\chi
\punt\alpha)\qquad\mbox{if } a \in R,
\nonumber\\[-8pt]\\[-8pt]
\mbox{\textit{Additivity}}\quad\chi\punt(\alpha+ \gamma) &\equiv& \chi\punt\alpha
\dotplus
\chi\punt\gamma.
\nonumber
\end{eqnarray}
The \textit{semi-invariance under translation} follows from both
equivalences in (\ref{propcum1}) by choosing
as umbra $\alpha$ the unity umbra $u$.

As done in (\ref{formalcumulants}), multivariate formal cumulants $\{
c_{\ibs}\}$ of a
sequence of multivariate moments $\{g_{\ibs}\}$ can be defined via
generating functions.
Indeed, if $\{g_{\ibs}\}$ is umbrally represented by the $m$-tuple
$\nubs$, then the sequence $\{c_{\ibs}\}$ is umbrally represented by
the $m$-tuple $\kappabs_{\nubs}$ such that
%
%
\begin{equation}\label{genfun}
f(\nubs,\zbs) = \exp\bigl[f(\kappabs_{\nubs},\zbs) - 1 \bigr].
\end{equation}
The $m$-tuple $\kappabs_{\nubs}$ is named $\nubs$-cumulant. By
comparing (\ref{multivariate1}) with (\ref{genfun}), the
following equivalence follows:
%
%
\begin{equation}\label{momcumumbral}
\nubs\equiv u \punt\beta\punt\kappabs_{\nubs}.
\end{equation}
Equivalence (\ref{momcumumbral}) can be inverted in
%
%
\begin{equation}\label{cummomumbral}
\kappabs_{\nubs} \equiv u^{\langle-1\rangle} \punt\beta\punt\nubs
\qquad\mbox{with } f(\kappabs_{\nubs},\zbs) = 1 +
\log\bigl[f(\nubs,\zbs)\bigr],
\end{equation}
where $u^{\langle-1\rangle}$ is the compositional inverse of the
unity umbra $u$.
Moments of $u^{\langle-1\rangle} \punt\beta\punt\nubs$ can
be computed
via equation (\ref{eq16}) by recalling (\ref{momcompinvu}). As before,
the umbra $u^{\langle-1\rangle} \punt\beta$ may be replaced by
the umbra $\chi$, so that
%
%
\begin{equation}\label{cummomumbral1}
\kappabs_{\nubs} \equiv\chi\punt\nubs.
\end{equation}
Thanks to this last representation, the algebraic properties of
cumulants can be formalized as
%
%
\begin{eqnarray}
\label{propcum}
\mbox{\textit{Homogeneity}}\quad\quad\hspace*{11pt}\chi\punt(a \nubs) &\equiv& a
(\chi\punt
\nubs)\qquad\mbox{if } a \in R,
\nonumber
\\
\mbox{\textit{Additivity}}\quad\chi\punt(\nubs_1 + \nubs_2)
&\equiv&\chi\punt\nubs_1 \dotplus\chi\punt\nubs_2\\
&&\eqntext{\mbox{if $\nubs_1$ and $\nubs_1$ are uncorrelated
$n$-tuples}.}
\end{eqnarray}
In the\vspace*{1pt} additivity property (\ref{propcum}), we have used the
disjoint sum of two $m$-tuples, that is,
$E[(\nubs_1 \dotplus\nubs_2)^{\ibs}]=E[\nubs^{\ibs}_1] + E[\nubs
^{\ibs}_2]$ for all $\ibs\in\mathbb{N}_0^m$.
The \textit{semi-invariance under translation} follows from both
equivalences in (\ref{propcum}) by choosing
the $m$-tuple $\ubs=(u,\ldots,u)$ as $\nubs$.
%
\subsection{Cumulants of trace powers}\label{sec4.2}
Let us represent the eigenvalues of a random matrix $M$ of dimension $m$
by the $m$-tuple of umbral monomials $\nubs=(\nu_1,\ldots, \nu_m)$.
Cumulants of $\nubs$ can be recovered via (\ref{momcumumbral}) and
(\ref{cummomumbral}).
In this section we will characterize cumulants of $\operatorname{Tr} (M)$,
that is, cumulants of the sequence $E[(\nu_1 + \cdots+\nu_m)^k]$ for
$k=1,2,\ldots\,$.
Observe that
%
%
\begin{equation}\label{cumtrace}
f(\nu_1 + \cdots+\nu_m,z)=E\bigl[e^{(\nu_1 + \cdots+\nu_m)z}
\bigr]=f(\nubs,\zbs)
\end{equation}
by using (\ref{gf}) and (\ref{genfun1}) with $\zbs=(z,\ldots,z)$.
Compositions of multivariate formal power series like $f(\nubs,\zbs)$
in (\ref{cumtrace}) are represented by symbols with a peculiar
expression. Indeed, if $\nubs$ is a $m$-tuple of umbral monomials with
generating function $f(\nubs,\zbs)$, and $\xibs$ is a $m$-tuple of
umbral monomials with generating function $f(\xibs,\zbs)$, the
$m$-tuple having generating function $f[\xibs,
(f(\nubs,\zbs)-1,\ldots,f(\nubs,\zbs)-1)]$ is \mbox{$(\xi_1 + \cdots+
\xi_m) \punt\beta\punt\nubs$}, that is,
%
%
\begin{equation}\label{multivariate2}\qquad
f\bigl[(\xi_1 + \cdots+ \xi_m) \punt\beta\punt\nubs,
\zbs\bigr] = f\bigl[\xibs, \bigl(f(\nubs,\zbs)-1,\ldots, f(\nubs,\zbs
)-1\bigr)
\bigr].
\end{equation}
As in (\ref{cumuniv}) and (\ref{fundcum}) for univariate and
multivariate cumulants, respectively, in
order to characterize cumulants of $f(\nubs,\zbs)$ in (\ref
{cumtrace}), we replace the $m$-tuple $\xibs$ with the $m$-tuple
$\ubs=(u,\ldots,u)$ in (\ref{multivariate2}).
Denote by ${\mathfrak c}_{\nubs}$ the $m$-tuple such that $\nubs
\equiv(u + \cdots+ u) \punt\beta\punt{\mathfrak c}_{\nubs}$, that is,
%
%
\begin{equation}\label{cumtrace1}
f(\nubs,\zbs)=f\bigl[(u + \cdots+ u) \punt\beta\punt{\mathfrak
c}_{\nubs}, \zbs\bigr]
\end{equation}
with $f(\nubs,\zbs)$ in (\ref{cumtrace}).
If ${\mathfrak c}_{\nubs} =({\mathfrak c}_{1,\nubs},\ldots,
{\mathfrak c}_{m,\nubs})$, then
$f({\mathfrak c}_{\nubs},\zbs) = f[{\mathfrak c}_{1,\nubs} + \cdots
+ {\mathfrak c}_{m,\nubs},z]$
and
\[
f\bigl[(u + \cdots+ u) \punt\beta\punt{\mathfrak c}_{\nubs}, \zbs
\bigr]=\exp\bigl\{m \bigl(f[{\mathfrak c}_{1,\nubs} + \cdots+ {\mathfrak
c}_{m,\nubs},z]-1\bigr)\bigr\}.
\]

%
\begin{defn}\label{firstdef}
For fixed $m$, formal cumulants of the sequence $\{E[(\nu_1 + \cdots
+\nu_m)^k]\}$ are umbrally represented by
the umbral polynomial ${\mathfrak c}_{1,\nubs} + \cdots+ {\mathfrak
c}_{m,\nubs}$, such that
%
%
\begin{equation}\label{cumtraceumbral}
\nu_1 + \cdots+\nu_m \equiv m \punt\beta\punt({
\mathfrak c}_{1,\nubs} + \cdots+ {\mathfrak c}_{m,\nubs})
\end{equation}
with ${\mathfrak c}_{\nubs} =({\mathfrak c}_{1,\nubs},\ldots,
{\mathfrak c}_{m,\nubs})$ given in
(\ref{cumtrace1}).
\end{defn}
In order to prove that the moments of the umbral polynomial ${\mathfrak
c}_{1,\nubs} +
\cdots+ {\mathfrak c}_{m,\nubs}$ satisfy the characterizing algebraic
properties of cumulants, we need to invert (\ref{cumtraceumbral}).
%
%
\begin{prop} \label{revert}
We have ${\mathfrak c}_{1,\nubs} + \cdots+ {\mathfrak c}_{m,\nubs}
\equiv\chi\punt\frac{1}{m} \punt(\nu_1 + \cdots+\nu_m)$.
\end{prop}
\begin{pf}
Indeed from (\ref{cumtraceumbral}), we have $\frac{1}{m} \punt(\nu
_1 + \cdots+\nu_m) \equiv\beta\punt({\mathfrak c}_{1,\nubs} +
\cdots+ {\mathfrak c}_{m,\nubs})$, so $\chi\punt\frac{1}{m} \punt
(\nu_1 + \cdots+\nu_m) \equiv\chi\punt\beta\punt({\mathfrak
c}_{1,\nubs} + \cdots+ {\mathfrak c}_{m,\nubs})$. The result follows
since $\chi\punt\beta\equiv u$.
\end{pf}
Thanks to Proposition~\ref{revert}, the umbral polynomial ${\mathfrak
c}_{1,\nubs} + \cdots+ {\mathfrak c}_{m,\nubs}$ is similar to an
umbra like $\chi\punt p$, with $p \in R[\A]$.
So it has to satisfy the additivity and homogeneity properties like
those in (\ref{propcum}).
%
%
\begin{theorem}
\textup{Additivity}: If $\nubs_1$ and $\nubs_2$ are uncorrelated
$m$-tuples, then
\[
{\mathfrak c}_{1,\nubs_1 + \nubs_2} + \cdots+ {\mathfrak c}_{m,\nubs
_1 + \nubs_2} \equiv({
\mathfrak c}_{1,\nubs_1} + \cdots+ {\mathfrak c}_{m,\nubs_1}) \dotplus({
\mathfrak c}_{1,\nubs_2} + \cdots+ {\mathfrak c}_{m,\nubs_2});
\]

\textup{Homogeneity}: if $a \in R$, then ${\mathfrak c}_{1,{(a
\nubs)}} + \cdots+ {\mathfrak c}_{m,{(a \nubs)}}
\equiv a ({\mathfrak c}_{1,{\nubs}} + \cdots+ {\mathfrak c}_{m,{\nubs}})$.
\end{theorem}
\begin{pf}
Observe that $a \punt(p + q) \equiv a \punt p + a \punt q$ if $c \in
R$ and $p,q$ are umbral polynomials with disjoint supports\vadjust{\goodbreak}
\cite{Bernoulli}. The previous equivalence holds in particular for
$a=1/m$ and $p=(\nu_{1,1} + \cdots+\nu_{1,m})$ and $q=(\nu_{2,1} +
\cdots+\nu_{2,m})$.
From Proposition~\ref{revert},
\[
{\mathfrak c}_{1,\nubs_1 + \nubs_2} + \cdots+ {\mathfrak c}_{m,
\nubs_1 + \nubs_2} \equiv\chi
\punt\frac{1}{m} \punt(p + q) \equiv\chi\punt\biggl[ \frac{1}{m}
\punt p + \frac{1}{m} \punt q \biggr].
\]
From the additivity property in (\ref{propcum}), we have
\[
\chi\punt\biggl[ \frac{1}{m} \punt p + \frac{1}{m} \punt q \biggr]
\equiv\chi\punt\frac{1}{m} \punt p \dotplus\chi\punt\frac
{1}{m} \punt
q
\]
for $p$ and $q$ umbral polynomials with disjoint supports. The result
follows by observing that ${\mathfrak c}_{1,\nubs} + \cdots+
{\mathfrak c}_{m,\nubs} \equiv
\chi\punt\frac{1}{m} \punt q$ and ${\mathfrak c}_{1,\nubs} +
\cdots+ {\mathfrak c}_{m,\nubs} \equiv
\chi\punt\frac{1}{m} \punt p$.
The homogeneity property follows since $b \punt(a p) \equiv a ( b
\punt p)$, for $a, b \in R$ and $p \in R[\A]$~\cite{Bernoulli}. The
previous equivalence holds
in particular for $a=1/m$ and $p=(\nu_1 + \cdots+\nu_m)$.
\end{pf}
The semi-invariance under translation follows since $({\mathfrak
c}_{1,\ubs} + \cdots+ {\mathfrak c}_{m,\ubs}) \equiv\chi\punt
\frac{1}{m} \punt m \equiv\chi$ whose moments are all zero except
the first. The connection between multivariate cumulants of $\nubs$
and those of $\nu_1 + \cdots+ \nu_m$ is given in the following proposition.
%
%
\begin{prop} We have $\kappabs_{\nubs} \equiv\chi\punt m \punt
\beta\punt{\mathfrak c}_{\nubs}$.
\end{prop}
\begin{pf}
We have $(u + \cdots+ u) \punt\beta\punt{\mathfrak c}_{\nubs}
\equiv m \punt\beta\punt{\mathfrak c}_{\nubs}$,
and from (\ref{cumtrace1}) we have $m \punt\beta\punt{\mathfrak
c}_{\nubs} \equiv\nubs$. Therefore
from (\ref{momcumumbral}), we have $u \punt\beta\punt\kappabs
_{\nubs} \equiv m \punt\beta\punt{\mathfrak c}_{\nubs}$ so that\break
$\chi\punt u \punt\beta\punt\kappabs_{\nubs} \equiv\chi\punt m
\punt\beta\punt{\mathfrak c}_{\nubs}$. The result follows since
$\chi\punt u \punt\beta\equiv u$ and $\chi\punt u \punt\beta\punt
\kappabs_{\nubs} \equiv\kappabs_{\nubs}$.
\end{pf}

\section{Spectral $k$-statistics}\label{sec5}

Tukey~\cite{Tukey1} introduced the multi-index $k$-statistics in connection
with finite-population sampling.
He showed that the multi-index $k$'s are multiplicative in the limit as
$n\to\infty$,
and that they are equal to the product of Fisher's single-index $k$'s.
In the ordinary i.i.d. setting considered by Fisher, this means that
each multi-index $k$ converges to a cumulant product.

We now construct matricial polykays, indexed by an integer partition
$\lambda$, as
unbiased estimators of cumulant products of trace powers of a random
matrix $Y$.
Then, when the random matrix $Y$ is defined by sub-sampling
as in Section~\ref{sec2}, that is, when a spectral sample is
considered, we will prove
the inheritance property by assuming that
the elements of the diagonal matrix $X$ are umbrally represented by
similar and uncorrelated
umbrae.

%
%
\begin{defn}\label{kstat}
The matricial polykay $\kappa_{\lambda}(\ybs)$ of class $\lambda
\vdash i$
is the symmetric polynomial in the eigenvalues $\ybs=(y_1,\ldots, y_m)$
such that
\[
E\bigl[\kappa_{\lambda}(\ybs)\bigr] = \prod_{j=1}^{l(\lambda)}
E\bigl[({\mathfrak c}_{1,\ybs} + \cdots+ {\mathfrak c}_{m,\ybs
})^{\lambda_j}
\bigr].\vadjust{\goodbreak}
\]
\end{defn}
Set $l(\lambda)=r$. From Proposition~\ref{revert} and equation (\ref
{eq15}), with $n$ replaced by the umbra $\chi$, a first expression
of cumulant products of trace powers in terms of moments of $Y$ is
%
%
\begin{eqnarray}\label{matricialpolykays}
&&
\prod_{j=1}^{r} E\bigl[({\mathfrak
c}_{1,\ybs} + \cdots+ {\mathfrak c}_{m,\ybs})^{\lambda_j}\bigr]
\nonumber\\[-8pt]\\[-8pt]
&&\qquad= \sum_{(\eta
_1 \vdash\lambda_1,\ldots,
\eta_r \vdash\lambda_r)} \prod_{j=1}^r
\frac{(-1)^{\nu_{\eta
_j}-1}}{m} d_{{\eta_j}} (\nu_{\eta_j}-1)! g_{\eta_1 + \cdots+
\eta_r},\nonumber
\end{eqnarray}
where $\eta_1 + \cdots+ \eta_r=(t_1, t_2, \ldots)$ is the summation
of the partitions $\{\eta_1,\ldots, \eta_r\}$
and $g_{\eta_1 + \cdots+ \eta_r} = \prod_{j=1}^{l(\eta_1 + \cdots
+ \eta_r)} E[\Tr(Y^{t_j})]$.
Equation (\ref{matricialpolykays}) takes into account that $E[(\chi
\punt\chi)^i]=(-1)^{i-1} (i-1)!$ for all nonnegative integers $i$.

A second expression which is more suitable for spectral sampling is in
terms of joint moments of $Y$, that is,
in terms of products of its trace powers. To this end, we need to work
with permutations
${\mathfrak S}_i$ and with the group algebra $R[\A]({\mathfrak S}_i)$
on the polynomial ring $R[\A]$.

A permutation $\sigma$ of $[i]$, or $\sigma\in{\mathfrak S}_i$, the
symmetric group,
can be decomposed into disjoint cycles $C(\sigma)$.
In the standard representation each cycle is written with its largest
element first,
and the cycles are listed in increasing order of their largest element
\cite{Stanley1}.
The length of the cycle $c \in C(\sigma)$ is its cardinality, denoted
by ${\mathfrak l}(c)$.
The number of cycles of $\sigma$ is denoted by $|C(\sigma)|$. Recall
that a
permutation $\sigma$ with $r_1$ $1$-cycles, $r_2$ $2$-cycles and so on
is said to
be of cycle class $\lambda=(1^{r_1},2^{r_2},\ldots) \vdash i$.
In particular we have $l({\lambda})=|C(\sigma)|$. The number
of permutations $\sigma\in{\mathfrak S}_i$ of cycle class $\lambda
=(1^{r_1}, 2^{r_2}, \ldots)
\vdash i$ is usually denoted by
%
%
\begin{equation}\label{numperpart}
s_{\lambda}= \frac{i!}{1^{r_1} r_1! 2^{r_2} r_2! \cdots}.
\end{equation}
Each cycle class is a conjugacy class of the group of permutations:
two elements of ${\mathfrak S}_i$ are conjugate if and only if they
have the same cycle class.

Consider the group algebra ${\mathbb A}_i=R[\A]({\mathfrak S}_i)$.
An element $f \in{\mathbb A}_i$ associates with each permutation
$\sigma\in{\mathfrak S}_i$ a polynomial $f(\sigma)\in R[\A]$,
so ${\mathbb A}_i$ is the space of $R[\A]$-valued functions.
Multiplication in ${\mathbb A}_i$ is the convolution
%
%
\begin{equation}\label{mult}
(f \cdot g) (\sigma) = \sum_{\rho\omega= \sigma} f(\rho) g(
\omega).
\end{equation}
The unitary element with respect to multiplication
is the indicator function $\delta$ such that $\delta(e) = 1$, with
$e$ the
identity $[i]\to[i]$, and zero otherwise.
Indeed $f \cdot\delta=\delta\cdot f=f$ for $f \in{\mathbb A}_i$.
If it exists, the inverse function of $f$
in ${\mathbb A}_i$ is denoted by $f^{(-1)}$ and is such that
$f^{(-1)} \cdot f = f \cdot f^{(-1)} = \delta$.

Denote by $\mu(Y)$ the function in ${\mathbb A}_i$ such that
%
%
\begin{equation}\label{defmu}
\mu(Y) (\sigma) = \prod_{c \in C(\sigma)} \Tr
\bigl(Y^{{\mathfrak
l}(c)} \bigr) \in R[\A]\vadjust{\goodbreak}
\end{equation}
for a matrix $Y$ of order $m$ and $\sigma\in{\mathfrak S}_i$.
Evidently $\mu(Y)(\sigma)$ is a product of power sums in the
eigenvalues of $Y$,
depending only on the cycle structure. In particular we have
%
%
\begin{equation}\label{ident}
\mu(Y) (e) = \bigl[\Tr(Y)\bigr]^i \quad\mbox{and}\quad\mu(I_m) (
\sigma) = m^{|C(\sigma)|}.
\end{equation}

%
%
\begin{theorem}\label{powsum}
Define the function $\tilde{\kappa}(\ybs) \in{\mathbb A}_i$ by
%
%
\begin{equation}\label{natural}
\tilde{\kappa}(\ybs) = \mu(I_m)^{(-1)}\cdot\mu(Y)
\end{equation}
with $\mu(Y)$ and $\mu(I_m)$ given in (\ref{defmu}) and (\ref{ident}),
respectively.
Then
%
%
\begin{equation}\label{kstat2}
{\mathfrak K}_{\lambda}(\ybs) \simeq\bigl[(1!)^{r_2}
(2!)^{r_3} \cdots\bigr] \tilde{\kappa}(\ybs) (\sigma)
\end{equation}
is a matricial polykay of class $\lambda=(1^{r_1}, 2^{r_2}, \ldots)
\vdash i$, the cycle structure of $\sigma$.
\end{theorem}

\begin{pf}
Observe that by taking the expectation of both sides in (\ref
{natural}), we have
%
%
\begin{equation}\label{Capitaine}
E\bigl[\tilde{\kappa}(\ybs)\bigr] = \mu(I_m)^{-1} \star E
\bigl[\mu(Y)\bigr],
\end{equation}
where $\star$ is the classical convolution on the space of
$R$-functions on ${\mathfrak S}_i$
\[
(a \star b) (\sigma)= \sum_{\rho\omega= \sigma} a(\rho) b(\omega)
\]
with $E[\mu(Y)]\dvtx \sigma\in{\mathfrak S}_i \mapsto E[\mu(Y)(\sigma
)] \in R$ and
$E[\tilde{\kappa}(\ybs)]\dvtx \sigma\in{\mathfrak S}_i \mapsto
E[\tilde{\kappa}(\ybs)(\sigma)] \in R$.
Then the symmetric polynomial ${\mathfrak K}_{\lambda}(\ybs)$ in
(\ref{kstat2}) satisfies
Definition~\ref{kstat} if also the function
\[
E[C_{\ybs}]\dvtx \sigma\in{\mathfrak S}_i \mapsto\prod
_{\sigma\in
C(\sigma)} \frac{1}{({\mathfrak l}(\sigma) -1)!} E\bigl[({\mathfrak
c}_{1,\ybs} + \cdots+ {\mathfrak c}_{m,\ybs})^{{\mathfrak l}(\sigma
)}\bigr]
\]
is such that $E[C_{\ybs}]=\mu(\Id_m)^{-1} \star E[\mu(Y)]$.
From (\ref{cumtraceumbral}), we have
%
%
\begin{equation}\label{momtocum}
E\bigl[\bigl(\Tr(Y)\bigr)^i\bigr]= \sum
_{\lambda\vdash i} d_{\lambda} m^{l(\lambda)} \prod
_{j=1}^{l(\lambda)} E\bigl[({\mathfrak c}_{1,\ybs} +
\cdots+ {\mathfrak c}_{m,\ybs})^{\lambda_i}\bigr].
\end{equation}
From (\ref{momtocum}), by observing that
%
%
\begin{equation}\label{dlambda}
d_{\lambda} = \frac{s_{\lambda}}{(1!)^{r_2} (2!)^{r_3} \cdots}
\end{equation}
with $s_{\lambda}$ the number of permutations $\sigma\in{\mathfrak
S}_i$ of
cycle class $\lambda=(1^{r_1}, 2^{r_2}, \ldots) \vdash i \leq m$
given in (\ref{numperpart}),
we have
%
%
\begin{equation}\label{final}\quad
E\bigl[\mu(Y)\bigr](e) = \sum_{\sigma\in{\mathfrak S}_i}
m^{|C(\sigma)|} E[C_{\ybs}](\sigma) = \sum
_{\sigma\in{\mathfrak S}_i} m^{|C(\sigma)|} E[C_{\ybs}]\bigl(
\sigma^{-1}\bigr),
\end{equation}
where\vspace*{1pt} $\sigma^{-1}$ is the inverse of $\sigma$. By
using the action of ${\mathfrak S}_i$ on the group algebra
$R({\mathfrak S}_i)$ we have
$E[\mu(Y)] = \mu(I_m) \star E[C_{\ybs}]$. For $i \leq m$ the
function $\sigma\mapsto m^{|C(\sigma)|}$
has an inverse~\cite{Letac}, so that $E[C_{\ybs}]=\mu(I_m)^{-1}
\star E[\mu(Y)]$.\vadjust{\goodbreak}
\end{pf}
Note that equation (\ref{Capitaine}) is given in~\cite{Capitaine}
as the definition of the cumulants of a random matrix.
By Theorem~\ref{powsum}, we have shown that $\tilde{\kappa}(\ybs)$
in (\ref{Capitaine})
are rather statistics, due also to the condition $i \leq m$ which
parallels the analogous condition for Fisher's $k$-statistics.

The statistics ${\mathfrak K}_{\lambda}(\ybs)$ are unbiased
estimators of product of cumulants, due to Definition
\ref{kstat}. The inheritance on the average is indeed strictly
connected to the spectral sampling, that is,
to the special structure of the matrix $Y = (H X H^\dag)_{[m \times m]}$.
When ${\mathfrak K}_{\lambda}(\ybs)$ refers to spectral sampling, we
call them spectral polykays.

%
%
\begin{defn}[(Natural spectral statistics)]
A statistic is said to be \textit{natural} relative to spectral sampling
if, for each $m \leq n$, the average value of $T_m$ over spectral
sub-samples $\ybs$ of $\xbs$ is equal to $T_n(\xbs)$,
%
%
\begin{equation}\label{natural1bis}
E \bigl(T_m(\ybs)|\xbs\bigr)=T_n(\xbs).
\end{equation}
\end{defn}

Theorem~\ref{inerit} states that the spectral $k$-statistics
${\mathfrak K}_{\lambda}(\ybs)$ are natural.

We first give a proposition which moves from Lemma 7.2 in \cite
{Bernoulli}. In this proposition, the evaluation operator $E[ \cdot
| \gammabs]$ deals the elements of
the $n$-tuple $\gammabs$ as they were constants. A formal definition of
$E[ \cdot| \gammabs]$ may be found in~\cite{DinardoOliva}.
%
%
\begin{prop} \label{bbb}
If $\{\gamma_1,\gamma_2,\ldots, \gamma_n\}$ are uncorrelated
umbrae similar to
the umbra $\gamma$, then
%
%
\begin{equation}\label{lemma72quater}
E\bigl[(\gamma_1 z_1 + \gamma_2
z_2 + \cdots+ \gamma_n z_n)^j |
\gammabs\bigr] = \sum_{\lambda\vdash j} d_{\lambda}
\kappa_{\lambda}(\gammabs) E[\tilde{\sigma}_{\lambda}],
\end{equation}
where $\gammabs= (\gamma_1, \gamma_2,\ldots, \gamma_n)$ and
$\tilde{\sigma}$ is the polynomial umbra whose moments are the power
sums in the indeterminates $\{z_1,z_2,\ldots, z_n\}$, that is,
$E[\tilde{\sigma}^i]= z_1^i + z_2^i + \cdots+ z_n^i$.
\end{prop}
\begin{pf}
The starting point is the result of Lemma 7.2 in~\cite{Bernoulli},
%
%
\begin{equation}\label{lemma72}
\chi\punt(\gamma_1 z_1 + \cdots+ \gamma_n
z_n) \equiv(\chi\punt\gamma) \tilde{\sigma},
\end{equation}
where $\{\gamma_1,\gamma_2,\ldots, \gamma_n\}$ are uncorrelated
umbrae similar to
$\gamma$. Equivalence (\ref{lemma72}) may be rewritten as
%
%
\begin{equation}\label{lemma72bis}
\chi\punt(\gamma_1 z_1 + \cdots+ \gamma_n
z_n) \equiv\biggl(\chi\punt\frac{1}{n} \punt(
\gamma_1 + \gamma_2 + \cdots+ \gamma_n)
\biggr) \tilde{\sigma}
\end{equation}
as $\gamma\equiv\frac{1}{n} \punt n \punt\gamma$.
Taking the dot-product with $\beta$ on both sides of (\ref
{lemma72bis}) gives
%
%
\begin{equation}\label{lemma72ter}
(\gamma_1 z_1 + \cdots+ \gamma_n
z_n) \equiv\beta\punt\biggl[ \biggl(\chi\punt\frac{1}{n}
\punt(\gamma_1 + \gamma_2 + \cdots+ \gamma_n)
\biggr) \tilde{\sigma} \biggr].
\end{equation}
The result follows by using (\ref{comp}) with $\gamma$ replaced by
the unity
umbra $u$ and by using Proposition~\ref{revert} and Definition \ref
{kstat}. In equation (\ref{comp}), the evaluation operator is intended
to be replaced by $E[ \cdot| \gammabs]$.
\end{pf}

%
%
\begin{theorem}[(Inheritance on the average)] \label{inerit}
The statistics ${\mathfrak K}_{\lambda}(\ybs)$ in (\ref{kstat2})
are inherited on the average, that is,
\[
E\bigl[{\mathfrak K}_{\lambda}(\ybs)|\xbs\bigr]={\mathfrak
K}_{\lambda}(\xbs),
\]
where $\ybs$ is a spectral random sample.
\end{theorem}
\begin{pf}
Since the trace is invariant under cyclic permutations, for a
nonnegative integer $i$, we have
\[
\Tr\bigl(Y^i\bigr) = \Tr\bigl[\bigl(H X H^\dag
\bigr)_{[m \times m]}^{i} \bigr] = \Tr\bigl[\bigl(X
H_{[m \times n]}^{\dag} H_{[m \times n]}\bigr)^{i} \bigr]
\]
with\vspace*{1pt} $Y$ given in Definition~\ref{Def2}. Therefore $\mu(Y) = \mu(X
B)$, with $B = H_{[m \times n]}^{\dag}$ $H_{[m \times n]}$ a square
matrix of dimension $n$ independent of $X$. The random matrix $B$ is an
orthogonal projection on a $m$-dimensional subspace such that
%
%
\begin{eqnarray}\label{orth}
\mu(B) (\sigma) & = & \prod_{c \in C(\sigma)} \Tr\bigl[
\bigl(H_{[m
\times n]}^{\dag} H_{[m \times n]}\bigr)^{l(c)}
\bigr]
\nonumber\\[-8pt]\\[-8pt]
& = & \prod_{c \in C(\sigma)} \Tr\bigl[\bigl(H_{[m \times n]}
H_{[m
\times n]}^{\dag}\bigr)^{l(c)} \bigr] =
\mu(I_m) (\sigma).
\nonumber
\end{eqnarray}
For a diagonal matrix $X$ independent of $B$, and by using Proposition
\ref{bbb}, we have
%
%
\begin{equation}\label{fund1}
E\bigl[\mu(X B) | \xbs\bigr] = \mu(I_m) \cdot\tilde{\kappa}(\xbs).
\end{equation}
Indeed, if in Proposition~\ref{bbb} the umbrae $\{\gamma_1,\ldots,
\gamma_n\}$ are replaced by the elements of the diagonal matrix $X$,
and the indeterminates $\{z_1,\ldots, z_n\}$ by the diagonal entries
of the matrix $B$, equation (\ref{lemma72quater}) may be updated as
%
%
\begin{equation}\label{lemma72ter}
E\bigl[\Tr(X B)^i | \xbs\bigr] = \sum_{\lambda\vdash i}
d_{\lambda} \kappa_{\lambda}(\xbs) E[ \tilde{\sigma}_{\lambda}].
\end{equation}
Due to (\ref{orth}), we have $E[ \tilde{\sigma}_{\lambda}] =
m^{l({\lambda})}$.
So again equation (\ref{fund1}) follows by using the action of
${\mathfrak S}_i$ on the group algebra $R({\mathfrak S}_i)$.
The result follows from Theorem~\ref{powsum} by observing that
\begin{eqnarray*}
E\bigl[\tilde{\kappa}(\ybs)| \xbs\bigr] &=& E\bigl[\mu(I_m)^{-1}
\cdot\mu(XB) | \xbs\bigr] \\
&=& E\bigl[\mu(I_m)^{-1}\cdot
\mu(I_m) \cdot\tilde{\kappa}(\xbs)| \xbs\bigr]=\tilde{\kappa}(\xbs).
\end{eqnarray*}
\upqed\end{pf}
%
%
\begin{rem}
The computation of $\mu(I_m)^{-1}$ requires the solution of a
system of
$m$ equations in $m$ indeterminates
$\mu(I_m) \cdot\mu(I_m)^{-1} = \mu(I_m)^{-1} \cdot\mu(I_m) =
\delta$ with coefficients
given by $\mu(I_m)$.
This task may be performed with standard procedures in any symbolic package.
A different way consists of resorting to the so-called Weingarten
function on ${\mathfrak S}_i$.
See~\cite{Collins} for the definition and the properties of the
Weingarten function,
which involves the characters of ${\mathfrak S}_i$
and Schur symmetric polynomials indexed by $\lambda\vdash i$.
\end{rem}

The spectral $k$-statistics can be expressed on terms of power sums
$S_r = \sum_{j=1}^n x_j^r$ as follows:
\begin{eqnarray*}
{\mathfrak K}_{(1)} &=& \frac{S_1}{n}= k_{(1)},
\\[-3pt]
{\mathfrak K}_{(2)} &=& \frac{n S_2-S_1^2}{ n ( n^2-1 )} =
\frac{k_{(2)}}{(n+1)},
\\[-3pt]
{\mathfrak K}_{(1^2)} &=& \frac{n S_1^{2} - S_2}{ n ( n^2-1
) } = \frac{k_{(1^2)}}{(n+1)},
\\[-3pt]
{\mathfrak K}_{(3)} &=& 2 \frac{2 S_1^{3} - 3 n S_1 S_2
+ n^{2} S_3}{ n
( n^2-1 ) ( n^2-4 )} = \frac{2 k_{(3)}} {
(n+1)(n+2)},
\\[-3pt]
{\mathfrak K}_{(1,2)} &=& \frac{-2 n S_3 + (n^{2} + 2) S_1
S_2 - n S_1^3}{n
( n^2-1 ) ( n^2-4 )} = \frac{2 k_{(1,2)} - n
k_{(1)} k_{(2)}} {(n+1)(n+2)},
\\[-3pt]
{\mathfrak K}_{(1^3)} &=& \frac{ S_1^3 (n^2 - 2) - 3 n S_1
S_2 + 4 S_3}{n
( n^2-1 ) ( n^2-4 )} = \frac{2 k_{(1^3)} - 3
k_{(1)} k_{(2)} + n(n+3)(k_{(1)})^3}{(n+1)(n+2)}.
\end{eqnarray*}
The functions of degree $4$ are a little more complicated,
\begin{eqnarray*}
{\mathfrak K}_{(4)} &=& 6 \frac{ S_4 (n^3 + n) - 4 S_1 S_3
(n^2+1) + S_2^2 (3 - 2 n^2) + 10
n S_1^2 S_2 - 5 S_1^4}{n^2 ( n^2-1 ) ( n^2-4
) ( n^2-9 )},
\\[-3pt]
{\mathfrak K}_{(1,3)} &=& 2 \frac{ - 3 n S_4 (n^2 + 1) + S_1 S_3
(12 + 3 n^2 + n^4) + S_2^2 (6 n^2
- 9)}{n^2 ( n^2-1 ) ( n^2-4 ) ( n^2-9
)}
\\[-3pt]
&&{} + \frac{ - 3 n S_1^2 S_2 (n^2 + 1) + 2 (2 n^2 - 3)
S_1^4}{n^2 ( n^2-1 ) ( n^2-4 ) ( n^2-9
)},
\\[-3pt]
{\mathfrak K}_{(2^2)} &=& \frac{ 2 S_4 (3 n - 2 n^3) + 4
S_1 S_3 (4 n^2 - 6) + S_2^2 (18
+ n^4 - 6 n^2)}{n^2 ( n^2-1 ) ( n^2-4 ) (
n^2-9 )}
\\[-3pt]
&&{} + \frac{- 2 n S_1^2 S_2 (n^2 + 6) + (n^2 + 6) S_1^4}{n^2
( n^2-1 ) ( n^2-4 ) ( n^2-9 )},
\\[-3pt]
{\mathfrak K}_{(1^2,2)} &=& \frac{ 10 n S_4 -4 S_1 S_3 (n^2
+1) + S_2^2 (n^2 + 6)
+ n S_1^2 S_2 (n^2 + 1)}{n ( n^2-1 ) ( n^2-4
) (
n^2-9 )}
\\[-3pt]
&&{} + \frac{(4 - n^2) S_1^4}{n ( n^2-1 ) ( n^2-4
) (
n^2-9 )},
\\[-3pt]
{\mathfrak K}_{(1^4)} &=& \frac{ - 30 n S_4 + 4 S_1 S_3 (4
n^2 - 6) + S_2^2 (3 n^2 + 18) + 6
n S_1^2 S_2 (4 - n^2)}{n^2 ( n^2-1 ) ( n^2-4
) ( n^2-9 )}
\\[-3pt]
&&{} + \frac{(6 - 8 n^2 + n^4) S_1^4}{n^2 ( n^2-1 )
( n^2-4 ) ( n^2-9 )}.
\end{eqnarray*}
For comparison purposes, all of the single-index functions $\mfk
_{(r)}(\xbs)$
and the $k$-statistics $k_{(r)}(\xbs)$ for $r \ge2$ are invariant
under translation:
$\mfk_{(r)}(\xbs-\bar{\xbs}) = \mfk_{(r)}(\xbs)$.
If the mean is zero, the fourth-order statistics are
\begin{eqnarray*}
(n)_4 k_{(4)} &=& n^2(n+1)S_4 -
3n(n-1)S_2^2,
\\[-3pt]
\mfk_{(4)} &\propto& n\bigl(n^2+1\bigr) S_4 -
\bigl(2n^2-3\bigr) S_2^2,
\end{eqnarray*}
showing that $\mfk_{(4)}$ is not a simple multiple of $k_{(4)}$.

For a spectral sample, the first few conditional variances and
covariances are
\begin{eqnarray*}\hspace*{-5pt}
\var\bigl(\mfk_{(1)}(\ybs)\given\xbs\bigr) &=& \mfk_{(2)}(\xbs) \biggl(
\frac1m - \frac1n \biggr),
\\[-3pt]\hspace*{-5pt}
\cov\bigl(\mfk_{(1)}(\ybs), \mfk_{(2)}(\ybs) \given\xbs\bigr)
&=& 2\mfk_{(3)}(\xbs) \biggl(\frac1m - \frac1n \biggr),
\\[-3pt]\hspace*{-5pt}
\var\bigl(\mfk_{(2)}(\ybs)\given\xbs\bigr) &=& 2\mfk_{(2^2)}(\xbs)
\biggl(\frac1{m^2-1} - \frac1{n^2-1}
\biggr)
\\[-3pt]\hspace*{-5pt}
&&{} + 2\mfk_{(4)}(\xbs) \frac{(n-m)(2m^2n^2 -
3n^2-3m^2-mn+3)}{nm(m^2-1)(n^2-1)},
\end{eqnarray*}
which are similar to the covariances of the corresponding $k$-statistics.

We now characterize the limiting behavior of spectral polykays.
To this end, we recall the notion of \textit{free cumulant} occurring within
noncommutative probability theory~\cite{Speicher}.
A noncommutative probability space is a pair $({\mathit A},\Phi)$,
where ${\mathit A}$ is a unital noncommutative algebra, and $\Phi\dvtx
{\mathit A} \rightarrow{\mathbb C}$ is a unital linear functional.
This gives rise to a sequence of multilinear functional $\{\Phi_i\}$
on ${\mathit A}$ via $\Phi_i(a_1,\ldots, a_i)=\Phi(a_1 \cdots a_i)$.

Let ${\mathcal NC}$ denote the lattice of all noncrossing partitions of
$[i]$. A noncrossing partition $\pi= \{B_1,B_2,\ldots,B_k\}$ of the
set $[i]$ is a partition such that if $1 \leq h < l < s < k \leq i$,
with $h, s \in B_n$ and $l,k \in B_{n^{ \prime}}$, then $n =
n^{\prime}$. For any noncrossing partition $\pi$ and $a_1,\ldots, a_i
\in{\mathit A}$ we set
\[
\Phi_{\pi}(a_1,\ldots, a_i)=\prod
_{B \in\pi} \Phi(a_{j_1} \cdots a_{j_s})\vspace*{-3pt}
\]
for $B=(j_1 < \cdots< j_s)$. Free cumulants are defined as multilinear
functionals such that
\[
c_{\pi}(a_1,\ldots, a_i)=\prod
_{B \in\pi} c_{|B|}(a_{j_1} \cdots
a_{j_s})\vspace*{-3pt}
\]
and\vspace*{-3pt}
\[
c_i(a_1,\ldots, a_i)=\sum
_{\pi\in{\mathcal NC}} {\mathfrak m}(\pi, 1_i)
\Phi_{\pi}(a_1,\ldots, a_i),
\]
where ${\mathfrak m}(\pi, 1_i)$ is the Moebius function on the lattice
of noncrossing partitions~\cite{Speicher}. The $i$th cumulants of $a$ is
$c_i=c_i(a,\ldots, a)$.

By using Proposition 6.1 in~\cite{Capitaine}, when $m$ goes to
infinity, the mean of the normalized spectral $k$-statistics $\tilde
{\kappa}_{\lambda}^{(N)}$ corresponding to $\lambda=(1^{r_1},
2^{r_2}, \ldots) \vdash i$
\[
\tilde{\kappa}_{\lambda}^{(N)}(\ybs):= m^{i-l({\lambda})} \tilde{
\kappa}_{\lambda}(\ybs)\vadjust{\goodbreak}
\]
tends toward the product of free cumulants $c_1^{r_1} c_2^{r_2} \cdots$
with $\tilde{\kappa}_{\lambda}(\ybs):=\tilde{\kappa}(\ybs
)(\sigma)$, given in
(\ref{natural}), and $\sigma$ a permutation of class $\lambda$.

\section{Generalized spectral polykays}\label{sec6}

The notion of generalized cumulant has been discussed by McCullagh
\cite{McCullagh} and involves set partitions. In umbral terms, if $\pi$
is a partition of $\{\mu_1, \mu_2,\ldots, \mu_i\}$, then the
generalized cumulant $\kappa_{\pi}$ is defined as~\cite{Bernoulli}
\[
E\bigl[(\chi\punt\mu)_{\pi}\bigr]=\kappa_{\pi} \qquad\mbox{with }
(\chi\punt\mu)_{\pi}=\prod_{B \in\Pi_i} (\chi\punt
\mu_{B}) \mbox{ and } \mu_B=\prod
_{j \in B} \mu_j.
\]
For example, if $i=5$ and $\pi=\{\{\mu_1, \mu_2\},\{\mu_3\},\{\mu
_4,\mu_5\}\}$, then
\[
E\bigl[(\chi\punt\mu)_{\pi}\bigr] = E\bigl[(\chi\punt
\mu_1 \mu_2) (\chi\punt\mu_3) (\chi\punt
\mu_4 \mu_5)\bigr] = \kappa^{1 2, 3, 4 5}
\]
using McCullagh's notation.
Generalized $k$-statistics are the sample version of the generalized
cumulants. The importance of generalized $k$-statistics stems from the
following properties: the generalized $k$-statistics are linearly
independent; every polynomial symmetric function can be expressed
uniquely as a linear combination of generalized $k$-statistics; any
polynomial symmetric function whose expectation is independent of $n$
can be expressed as linear combination of
generalized $k$-statistics with coefficients independent of $n$ \cite
{McCullagh1}. Due to the last property,
natural statistics could be expressed as linear combinations of their
generalized $k$-statistics with coefficients independent of $n$.

%
%
\begin{theorem} \label{polyk}
If $\lambda\vdash i \leq m$ and $\pi$ is a set partition of class
$\lambda$, then generalized $k$-statistics
of spectral polykays are
%
%
\begin{equation}\label{defpoll}
\tilde{l}_{\pi}(\ybs) \simeq\sum_{\tau\geq\pi}
(-1)^{|\tau|-1} \bigl(|\tau|-1\bigr)! \tilde{\kappa}_{\tau}(\ybs),
\end{equation}
where $\tilde{\kappa}_{\tau}(\ybs)$ denotes the function on $\Pi
_i$ such that
$\tilde{\kappa}(\ybs)(\tau):= \tilde{\kappa}(\ybs)(\sigma)$,
with $\sigma\in{\mathfrak S}_i$ a permutation of the same class of
$\tau$
and $\tilde{\kappa}(\ybs)$ given in (\ref{natural}).
\end{theorem}
The proof of Theorem~\ref{polyk} relies on Proposition 5.4 of \cite
{Bernoulli}. We do not invert
equivalence (\ref{defpoll}) because the linear combination giving
spectral polykays in terms of their generalized $k$-statistics is quite
cumbersome; see equation (3.18) in~\cite{McCullagh1}. Instead, there
are alternative systems of symmetric functions that are more suitable from
a computational point of view. All such systems are invertible linear
functions of generalized $k$-statistics
with coefficients independent of the sample size, and the properties
given above are preserved under such transformations.

To characterize such coefficients, we first recall the Moebius
inversion formula on the lattice of set partitions~\cite{Rota}. The
set $\Pi_i$ with the refinement order $\leq$ is a lattice, where $\pi
\leq\tau$ if for any block\vadjust{\goodbreak}
in $B \in\pi$ there exists a block $B^{\prime} \in\tau$ such that
$B \subseteq B^{\prime}$. If $G$ is a function on $\Pi_i$ and
\[
F(\pi)=\sum_{\tau\geq\pi} G(\tau),
\]
then the Moebius inversion formula states that
%
%
\begin{equation}\label{moebius}
G(\pi)=\sum_{\tau\geq\pi} {\mathfrak m}(\pi, \tau) F(\tau),
\end{equation}
where ${\mathfrak m}(\pi, \tau)$ is the so-called Moebius function.
It is shown that
\[
{\mathfrak m}(\pi, \tau) = (-1)^{s-t} (2!)^{r_3}
(3!)^{r_4} \cdots,
\]
where $r_1+ 2 r_2 + \cdots= s = |\pi|$, $r_1+r_2+\cdots= t = |\tau
|$ and $(1^{r_1}, 2^{r_2}, \ldots)$ is the partition, usually denoted
by $\lambda(\pi, \tau)$,
of the integer $s$ such that $r_j$ blocks of $\tau$ contain exactly
$j$ blocks of $\pi$.
%
%
\begin{defn} \label{polyk1}
If $\lambda\vdash i \leq m$ and $\pi$ is a set partition of class
$\lambda$, the (transformed) generalized $k$-statistics of spectral
polykays are $l_{\tau}(\ybs)$ such that
%
%
\begin{equation}\label{kstatpolyk}
\tilde{\kappa}_{\pi}(\ybs) = \sum_{\tau\geq\pi}
l_{\tau}(\ybs),
\end{equation}
where $\tilde{\kappa}_{\tau}(\ybs)$ denotes the function on $\Pi
_i$ such that
$\tilde{\kappa}(\ybs)(\tau):= \tilde{\kappa}(\ybs)(\sigma)$,
with $\sigma\in{\mathfrak S}_i$ a permutation of the same class of
$\tau$ and $\tilde{\kappa}(\ybs)$ given in (\ref{natural}).
\end{defn}
The linear combination in (\ref{kstatpolyk}) is very simple
involving coefficients all equal to~$1$.
By using the Moebius inversion formula (\ref{moebius}), from (\ref
{kstatpolyk}) we have\looseness=-1
%
%
\begin{equation}\label{defpol}
l_{\pi}(\ybs)=\sum_{\tau\geq\pi} {\mathfrak m}(
\pi,\tau) \tilde{\kappa}_{\tau}(\ybs).
\end{equation}\looseness=0

Since $\tilde{\kappa}(\ybs)(\sigma)$ depends only on the cycle
structure $C(\sigma)$,
then $\tilde{\kappa}_{\tau}(\ybs)$ depends only on the block sizes
in $\tau$.
So in the sum (\ref{defpol}), there are $d_{\lambda}$ spectral
$k$-statistics equal
to $\tilde{\kappa}_{\tau}(\ybs)$, all those having the same class
$\lambda$.
Therefore spectral polykays of degree $i$ can be indexed by partitions
of $i$.
As example, by Definition~\ref{polyk}, the spectral polykays up to
order $4$ are
\begin{eqnarray*}
l_{(1)} & = & \tilde{\kappa}_{(1)}\qquad(i=1),
\\
l_{(1^2)} & = & \tilde{\kappa}_{(1^2)} - \tilde{
\kappa}_{(2)} \qquad(i=2),
\\
l_{(1,2)} & = & \tilde{\kappa}_{(1,2)} - \tilde{
\kappa}_{(3)} \qquad(i=3),
\\
l_{(1^3)} & = & \tilde{\kappa}_{(1^3)} - 3 \tilde{
\kappa}_{(1,2)} + 2 \tilde{\kappa}_{(3)},
\\
l_{(1,3)} & = & \tilde{\kappa}_{(1,3)} - \tilde{
\kappa}_{(4)} \qquad(i=4),
\\
l_{(2^2)} & = & \tilde{\kappa}_{(2^2)} - \tilde{
\kappa}_{(4)},
\\
l_{(1^2,2)} & = & \tilde{\kappa}_{(1^2,2)} - 2 \tilde{\kappa
}_{(1,3)} - \tilde{\kappa}_{(2^2)} + 2 \tilde{
\kappa}_{(4)},
\\
l_{(1^4)} & = & \tilde{\kappa}_{(1^4)} - 6 \tilde{
\kappa}_{(1^2,2)} + 8 \tilde{\kappa}_{(1,3)} + 3 \tilde{
\kappa}_{(2^2)} - 6 \tilde{\kappa}_{(4)}.
\end{eqnarray*}

In addition, we have $l_{(i)} = \tilde{\kappa}_{(i)}$.
We take a moment to motivate this definition.
Tukey~\cite{Tukey1} gives very similar equations connecting
classical polykays and $k$-statistics. We just recall those up to order $4$.
\begin{eqnarray*}
k_{(1)} & = & k_{(1)} \qquad(i=1),
\\
k_{(1^2)} & = & k_{(1)} k_{(1)} - {\frac{1}{m}}
k_{(2)} \qquad(i=2),
\\
k_{(1,2)} & = & k_{(1)} k_{(2)} - {\frac{1}{m}}
k_{(3)} \qquad(i=3),
\\
k_{(1^3)} & = & k_{(1)} k_{(1)} k_{(1)} - {
\frac{3}{m}} k_{(2)} k_{(1)} + {\frac{2}{m^2}}
k_{(3)},
\\
k_{(1,3)} & = & k_{(1)} k_{(3)} - {\frac{1}{m}}
k_{(4)} \qquad(i=4),
\\
k_{(2^2)} & = & {\frac{m-1}{m+1}} k_{(2)} k_{(2)} -
{\frac{1}{m}} k_{(4)},
\\
k_{(1^2,2)} & = & k_{(2)} k_{(1)} k_{(1)} - {
\frac
{2}{m}} k_{(3)} k_{(1)} - {\frac{m-1}{m(m+1)}}
k_{(2)} k_{(2)} + {\frac{2}{m(m+1)}} k_{(4)},
\\
k_{(1^4)} & = & k_{(1)} k_{(1)} k_{(1)}
k_{(1)} - {\frac
{6}{m}} k_{(2)} k_{(1)}
k_{(1)} + {\frac{8}{m^2}} k_{(3)} k_{(1)}\\[-3pt]
&&{}  + {
\frac{3(m-1)}{m^2(m+1)}} k_{(2)} k_{(2)}
- {\frac{6m}{m+1}} k_{(4)}.
\end{eqnarray*}

The two sets of equations are very similar in structure.

The refinement order in (\ref{kstatpolyk}) is inverted with respect
to those
connecting moments and cumulants~\cite{McCullagh}.
It is the same as that employed in the change of basis between augmented
symmetric functions and power sums~\cite{Bernoulli} and employed by
Tukey in order to show the multiplicative structure of $\tilde
{{\mathfrak a}}_{\lambda}(\xbs)$ for infinite populations.

In terms of power sums in the eigenvalues, the transformed generalized
spectral polykays up to order $4$ are
\begin{eqnarray*}
l_{(1, 2)} & = & \frac{(m+1) S_1 S_2 - S_1^3 - m
S_3}{m(m-1)(m+1)(m-2)},
\\
l_{(1^2, 2)} & = & \frac{2 m S_4 + (m+3) S_1^2 S_2 - (2m + 2) S_1
S_3 - m S_2^2 - S_1^4}{m(m-1)(m+1)(m-2)(m-3)},
\\
l_{(2^2)} & = & \frac{1}{m^2 ( m^2-1 ) ( m-2
) ( m-3 )} \\
&&{}\times\bigl\{ {S_{{1}}}^{4}+
\bigl( -3 m+3+{m}^{2} \bigr) {S_{{2}}}^{2}
\\
&&\hspace*{16.3pt}{} + ( 4 m-4 ) S_{{1}}S_{{3}} - 2 {S_{{1}}}^{2}
S_{{2}} m + \bigl( -{m}^{2}+m \bigr) S_{{4}} \bigr
\},
\\
l_{(1, 3)} & = & \frac{2}{{( m^2-4)(m^2-1) {m}^{2}}} \\[-3pt]
&&{}\times\bigl\{ -S_{{4}} m
\bigl(m^2 + 1\bigr) + S_{{1}} S_{{3}}
\bigl({m}^{3} +m^2 + 4\bigr)
\\
&&\hspace*{16.3pt}{} + S_2^2 \bigl(2 m^2 - 3\bigr) - m
S_1^2 S_2 (3 m + 1) + S_{{1}}^{4}
( 2 m - 1) \bigr\}.
\end{eqnarray*}
The spectral statistics $l_{(1^r)}$ are the same as the corresponding
polykays $k_{(1^r)}$.

%
\begin{theorem}
When $m$ goes to infinity the mean of the normalized (transformed)
generalized $k$-statistics
\[
l^{(N)}_{\pi}(\ybs):= m^{i-|\pi|} l_{\pi}(
\ybs)
\]
tends to $d_{\lambda} c_1^{r_1} c_2^{r_2} \cdots\,$, with
$\lambda=(1^{r_1}, 2^{r_2}, \ldots) \vdash i$, the class partition of
$\pi$, and $\{c_j\}$ free cumulants.
\end{theorem}
\begin{pf}
After multiplying both sides of (\ref{defpol}) by $m^{i-|\pi|}$, we have
\[
m^{i-|\pi|} l_{\pi}(\ybs) = \sum_{\tau\geq\pi}
{\mathfrak m}(\pi,\tau) m^{i-|\pi|} \tilde{\kappa}_{\tau}(\ybs).
\]
Since $\tau\geq\pi$, then $|\tau| \leq|\pi|$ so that
\[
l^{(N)}_{\pi}(\ybs) = \sum_{\tau\geq\pi}
{\mathfrak m}(\pi,\tau) \frac{1}{m^{|\pi|-|\tau|}} m^{i-|\tau|} \tilde
{\kappa
}_{\tau}(\xbs).
\]
As $m$ goes to infinity, for all $\tau> \pi$ having the same class
partition, $m^{i-|\tau|} \tilde{\kappa}_{\tau}(\ybs)$ tends toward
the free cumulant $c_{\tau}$, whereas $\frac{1}{m^{|\pi|-|\tau|}}$
goes to zero. The result\vspace*{1pt} follows since for $\tau= \pi$ we have
${\mathfrak m}(\pi,\tau)=1$, and for all
$\pi$ having the same class partition $m^{i-|\pi|} \tilde{\kappa
}_{\pi}(\ybs)$ goes to $c_{\pi}=c_1^{r_1} c_2^{r_2} \cdots\,$.
\end{pf}



%

\printaddresses

\end{document}